\documentclass{amsart}

\usepackage{amssymb,color}
\usepackage{amsfonts}
\usepackage{amsmath}
\usepackage{euscript}
\usepackage{enumerate}
\usepackage{graphics}
\usepackage{pdfsync}
\newtheorem{theorem}{Theorem}[section]
\newtheorem{lemma}[theorem]{Lemma}
\newtheorem{note}[theorem]{Note}
\newtheorem{prop}[theorem]{Proposition}
\newtheorem{cor}[theorem]{Corollary}

\newtheorem*{Theorem1'}{Theorem 1'}

\theoremstyle{definition}

\theoremstyle{remark}

\numberwithin{equation}{section}

\newcommand \g{{\mathfrak g}}
\newcommand \h{{\mathfrak h}}
\newcommand  \s{{\mathfrak s}}
\renewcommand \r{{\mathfrak r}}

\renewcommand \a{{\mathfrak a}}
\newcommand \e{{\mathfrak e}}
\def \n{{\mathfrak n}}

\newcommand \E{{\mathcal E}}

\newcommand \ad{{\mathrm{ad}}}
\newcommand \Hom{{\mathrm {Hom}}}
\newcommand \End{{\mathrm {End}}}

\newcommand \socl{{\mathrm {soc}}}
\newcommand \gl{{\mathfrak {gl}}}
\renewcommand \sl{{\mathfrak {sl}}}
\newcommand \so{{\mathfrak {so}}}
\newcommand \sy{{\mathfrak {sp}}}
\newcommand \C{{\mathbb C}}
\newcommand \Z{{\mathbb Z}}
\newcommand \Ze{{\mathcal Z}}
\newcommand \B{{\mathcal B}}

\newcommand \al{{\alpha}}
\newcommand \be{{\beta}}
\newcommand \ga{{\gamma}}
\newcommand \si{{\sigma}}

\newcommand \chr{{\mathrm {char}}}
\newcommand \NA{{\mathbf A}}
\newcommand \NE{{\mathbf E}}
\newcommand \NT{{\mathbf T}}

\begin{document}

\title[Classification of indecomposable modules over solvable Lie algebras ] {\small Classification of linked indecomposable modules
of a family of solvable Lie algebras over an arbitrary field of
characteristic 0}

\author{Leandro Cagliero}
\address{CIEM-CONICET, FAMAF-Universidad Nacional de C\'ordoba, C\'ordoba, Argentina.}
\email{cagliero@famaf.unc.edu.ar}
\thanks{The first author was supported in part by CONICET and SECYT-UNC grants.}

\author{Fernando Szechtman}
\address{Department of Mathematics and Statistics, University of Regina, Canada}
\email{fernando.szechtman@gmail.com}
\thanks{The second author was supported in part by an NSERC discovery grant}





\begin{abstract} Let $\g$ be a finite dimensional Lie
algebra over a field of characteristic 0, with solvable radical
$\r$ and nilpotent radical $\n=[\g,\r]$. Given a finite
dimensional $\g$-module $U$, its nilpotency series
$0\subset U^{(1)}\subset\cdots\subset U^{(m)}=U$ is defined so that
$U^{(1)}$ is the 0-weight space of $\n$ in $U$, $U^{(2)}/U^{(1)}$
is the 0-weight space of $\n$ in $U/U^{(1)}$, and so on. 
We say
that $U$ is linked if each factor of its nilpotency series is a
uniserial $\g/\n$-module, i.e., its $\g/\n$-submodules form a
chain. Every uniserial $\g$-module is linked, every linked
$\g$-module is indecomposable with irreducible socle, and both
converses fail.

In this paper we classify all linked $\g$-modules when $\g=\langle
x\rangle\ltimes \a$ and $\ad\, x$ acts
diagonalizably on the abelian Lie algebra $\a$. Moreover, we
identify and classify all uniserial $\g$-modules amongst them.
\end{abstract}

\maketitle

\section{Introduction}

We fix throughout an arbitrary field $F$. All vector spaces,
including all algebras, associative or Lie, as well as their
modules, are assumed to be finite dimensional over $F$. The only
exception to this rule is our use of polynomial algebras, although
all of their modules are still assumed to be finite dimensional.

A utopian ideal of the representation theory of Lie algebras is to
classify all indecomposable modules up to isomorphism. This is
exceedingly difficult to achieve, even for a Lie algebra as
elementary as the 2-dimensional abelian Lie algebra, as indicated
in the classical reference \cite{GP}.

For complex semisimple Lie algebras there is a completely
satisfactory answer. Thanks to Weyl's theorem, the building blocks
of the representation theory are the irreducible modules, which
can be classified by their highest weights (see \cite{H}).

For an arbitrary complex Lie algebra, the atoms of representation
theory are the indecomposable modules, and there is a large number
of recent articles devoted to their study, specifically to
classify the imbeddings of a given complex Lie algebra~$\g$ into a
given semismple Lie algebra $\s$, and to use this information to
produce and classify indecomposable $\g$-modules obtained by
restricting irreducible $\s$-modules to~$\g$. The origin of these
papers can be traced to an article by Repka and de Guise
\cite{RG}, where they devise a graphical method to furnish
indecomposable modules of the complexified Euclidean Lie
algebra $\e(2)$, which they use to find the decomposition of the tensor
product of some indecomposable $\e(2)$-modules.

The graphical approach of \cite{RG} was taken up by Douglas and
Premat \cite{DP}, who imbedded $\e(2)$ into $\sl(3)$ and produced
indecomposable $\e(2)$-modules by restricting irreducible
$\sl(3)$-modules. This was generalized by Premat \cite{Pr}, who
imbedded $\e(2)$ into every simple Lie algebra $\s$ of rank 2 and
showed that an irreducible $\s$-module restricts to an
indecomposable $\e(2)$-module. At the same time, it was shown by
Savage \cite{S} that $\e(2)$ is of wild representation type, which
somewhat explained the very large families of indecomposable
$\e(2)$-modules being produced. An extension of Savage's result
was given by Makedonskyi \cite{M} to virtually every Lie algebra
that is not semisimple or 1-dimensional.

The work of Douglas and Premat \cite{DP} inspired a paper by
Casati, Minniti and Salari \cite{CMS} on indecomposable
representations of the Diamond Lie algebra~${\mathcal{D}}$, a
4-dimensional subalgebra of the semidirect product $\sl(2)\ltimes
\h(1)$, where $\h(1)$ stands for 3-dimensional Heisenberg algebra.
They imbed ${\mathcal{D}}$ in $\sl(3)$, $\sy(4)$ and the truncated
Lie algebra $\sl(2)\ltimes \C[t]/(t^2)$ to provide various
families of indecomposable ${\mathcal{D}}$-modules. Like results
for ${\mathcal{D}}$ can be found in \cite{DJR}.

A study by Jakobsen of indecomposable representations of the
10-dimensional Poincar$\mathrm{\acute{e}}$ group, the group of
isometries of Minkowski spacetime, and its Lie algebra, with
references to applications in mathematical physics, can be found
in \cite{J}. Imbeddings of the Poincar$\mathrm{\acute{e}}$ algebra
into simple Lie algebras of rank 3 are given in \cite{DGR}, which
uses some of these imbeddings to produce families of
indecomposable modules for the Poincar$\mathrm{\acute{e}}$
algebra.

In addition to the above, many other families of indecomposable
modules for non-semisimple Lie algebras $\g$ have been produced by
imbedding them into suitable semisimple Lie algebras and
restricting irreducible modules of the latter to the former. See
\cite{DG}, \cite{DR}, \cite{DR2} and \cite{DR2} when $\g$ is the
complexified Euclidean Lie algebra $\e(3)$, \cite{C} when $\g$ is abelian, and
\cite{DKR} when $\g=\so(2n)\ltimes V$, where $V$ is a specially
chosen irreducible $\so(2n)$-module.

Rather than attempting to classify all indecomposable modules for
a given Lie algebra, or family of Lie algebras, we propose to
restrict the classification to certain types of structurally
defined indecomposable modules. The uniserial modules, namely
those possessing a unique composition series, seem to be
well-suited for this purpose. Our first contribution in this
regard is \cite{CS}. Much of \cite{CS} is concerned with the
construction of uniserial $\g$-modules when $\g=\s\ltimes V$,
where $\s$ is a complex semisimple Lie algebra and $V$ is a
non-trivial irreducible $\s$-module. In the special case
$\g=\s\ltimes \s_{\mathrm{ad}}$, \cite{CS} produces the
Kirillov-Reshetikhin modules from \cite{CM}, which are the subject
of active current interest (see \cite{BCG} and references
therein). Another significant part of \cite{CS} is devoted to
demonstrate that the problem of finding all uniserial
$\sl(2)\ltimes V(m)$-modules, $m\geq 1$, is equivalent to
determining all non-trivial zeros of the Racah-Wigner $6j$-symbol
within certain parameters. Actually locating these zeros, a
problem that in general is notoriously difficult (see \cite{R} and
\cite{L}), allowed us to conclude that no other uniserial
$\sl(2)\ltimes V(m)$-modules existed besides the ones we had
previously constructed.

A thorough analysis of the indecomposable modules of the complex
Lie algebra $\g=\sl(2)\ltimes V(1)$ was carried out by Piard
\cite{P}. In particular, he classifies all $\g$-modules $U$ such
that $U/\mathrm{rad}(U)$ is irreducible. The class of these
modules, which he called ``cyclic", is far greater than the class
of uniserial modules, but very small in comparison with the class
of all indecomposable modules. We have verified \cite{CS3} that
all of Piard's ``cyclic" modules can be constructed by linking
uniserial $\g$-modules in a suitable way. The process can be best
described by means of Young diagrams.

Thus uniserial modules are not only of interest in their own
right, but they can also be used to construct, understand and
classify more general indecomposable modules. The present paper
gives further evidence in this direction.

Let $\g$ be a Lie algebra with solvable radical $\r$ and let $U$ be a $\g$-module.
We say that $U$ is admissible if $U$ is non-zero and $[\g,\r]$ acts on $U$ via nilpotent operators.
If $\chr(F)=0$ then $[\g,\r]$ acts trivially on every irreducible, i.e., every non-zero $\g$-module is admissible
(see \cite{B}, I.5.3, or Lemma \ref{nilradical} below)
and, following \cite{B}, we refer to $[\g,\r]=[\g,\g]\cap\r$ as the nilpotent radical of $\g$ (equality follows from
a Levi decomposition of $\g$). In prime characteristic $[\g,\r]$ need not act trivially on an irreducible $\g$-module
(see \cite{H}, \S 4.3, Exercise 3). Whatever the characteristic of $F$, let $\n=[\g,\r]$.

If $U$ is an admissible $\g$-module we define its nilpotency series
$$
0=U^{(0)}\subset U^{(1)}\subset\cdots\subset U^{(m)}=U,
$$
so that the $i$-th nilpotency factor
$U^{(i)}/U^{(i-1)}$ of $U$, $1\leq i\leq m$, is the 0-weight space for the action of $\n$
on $U/U^{(i-1)}$. We refer to $m$ as the nilpotency length of~$U$. Note that $U^{(i)}/U^{(i-1)}$
is a module over $\g/\n$.

We introduce the notion of linked $\g$-module, by which we mean an admissible $\g$-module
all of whose nilpotency factors are uniserial as modules over $\g/\n$. Observe that every
admissible uniserial $\g$-module is linked and every linked $\g$-module
is indecomposable with irreducible socle, but both converses
fail: the $\n$-module $F^3$ given by
$\n=\left\{\left(\begin{smallmatrix}
0 & 0 & x \\ 0 & 0 & y \\ 0 & 0 & 0
\end{smallmatrix}\right):x,y\in F\right\}$
is indecomposable but not linked;
see Theorem \ref{main} for linked modules that are not uniserial.

Note that if $F$ has characteristic 0 and $\g$ is perfect
then the solvable and nilpotent radicals of $\g$ coincide and, consequently,
every linked $\g$-module is uniserial. However, if $\chr(F)=0$ and $\g$ is arbitrary,
the class of linked modules is far greater than the class of uniserial modules,
which is vastly greater than that of irreducible ones. Thus,
the path from a linked module down to its irreducible constituents is, in general, complicated.
In spite of this, we are able to classify all linked modules over a family of 2-step
solvable Lie algebras in characteristic 0.

Indeed, the goals of this paper are to classify all linked $\g$-modules
when $\g=\langle x\rangle\ltimes \a$ and $\ad\, x$ acts
diagonalizably on the abelian Lie algebra $\a$, and to identify
and classify the uniserial $\g$-modules amongst them. These
classifications are achieved under the only assumption that $F$
have characteristic 0. Note that even the most favourable case when $F$ is algebraically closed
of characteristic 0, every such $\g$ is shown in \cite{M} to be of wild representation type.

The construction of a family of linked $\g$-modules of nilpotency length
$m>1$ is given in Theorem \ref{falta} under no restrictions on $F$
whatsoever, while the fact that any linked $\g$-module of nilpotency length
$m>1$ is isomorphic to one and only one member of our family,
provided $\chr(F)=0$, can be found in Theorem \ref{main}.

In general, if $\g$ is a solvable Lie algebra then $\n=[\g,\g]$,
so a linked $\g$-module of nilpotency length 1 is just a uniserial module for the abelian Lie
algebra $\g/\n$.

The uniserial modules of an abelian Lie algebra are studied in
great detail in \cite{CS2}. A classification is
only achieved under certain conditions on $F$, too technical to
mention here, but certainly satisfied by any field of
characteristic~0.

This paper is part of a project attempting to identify a type of indecomposable module,
as general as possible, which admits classification for all Lie algebras.
The results obtained in \cite{CS,CS2} provide evidence in favor of the class of uniserial modules.
The solvable Lie algebra $\g=\langle x\rangle\ltimes \a$ considered in this paper is not only
the starting point to study the class of linked modules, but it is also important to our project
as it appears as a Lie subalgebra of many other families of Lie algebras.

Throughout the entire paper we let $F$ have arbitrary characteristic,
with the only exception of Lemma \ref{nilradical}, which is only valid when $\chr(F)=0$.
Our use of admissible modules allows us to treat the prime and zero characteristic cases
simultaneously, although we only obtain partial results in the former case.
Further information on the prime characteristic case can be found in \cite{CS4}.




\section{Indecomposable Lie algebras of transformations}
\label{trans}

Consider the group homomorphism $\si:F^+\to\mathrm{Aut}(F[X])$ given by
\begin{equation}
\label{actf}
g^{\si(\ga)}(X)=g(X+\ga),\quad \ga\in F,\, g\in F[X].
\end{equation}
Let $p_1\in F[X]$ be a monic irreducible polynomial of degree $d$. Given $m>1$, define $p_2,\dots,p_m\in F[X]$ by
\begin{equation}
\label{jio}
p_i(X)=p^{\si(i-1)}(X)=p_1(X+(i-1)),\quad 1\leq i\leq m.
\end{equation}
Let $s_1\geq\cdots\geq s_m$ be a decreasing sequence of natural numbers. We use these data to construct a torsion $F[X]$-module
\begin{equation}
\label{ww1}
U=F[X]u_1\oplus\cdots\oplus F[X]u_m=U_1\oplus\cdots\oplus U_m,
\end{equation}
where
\begin{equation}
\label{ww2}
\mathrm{Ann}_{F[X]}(u_i)=\{f\in F[X]|f(X)u_i=0\}=(p_i^{s_i}).
\end{equation}
Note that
$$
\dim_F U=t=d(s_1+\cdots+s_m).
$$
Given $\ga\in F$, we use the notation $U^{\sigma(\ga)}$ when viewing $U$ as an
$F[X]$-module via
$$
g(X)\cdot u=g^{\si(\ga)}(X)u=g(X+\ga)u.
$$
Note that
$$
\mathrm{\End}_{F[X]}(U)=\mathrm{\End}_{F[X]}(U^{\sigma(\ga)}),\quad
\ga\in F.
$$
Let $A\in \End_F(U)$ be multiplication by $X$, i.e.,
\begin{equation}
\label{ww3}
Au=Xu,\quad u\in U,
\end{equation}
and let $T\in\mathrm{\Hom}_{F[X]}(U^{\si(1)},U)$ be defined by
\begin{equation}
\label{defTT} Tu_1=0,\quad Tu_i=p_{i-1}(X)^{s_{i-1}-s_i}u_{i-1},\quad
2\leq i\leq m.
\end{equation}
Thus $T \in \End_F(U)$ satisfies
\begin{equation}
\label{calct}
Tg(X+1)u=g(X)Tu,\quad g\in F[X], u\in U.
\end{equation}
In particular,
$$
T(A+I)=AT,
$$
which means
\begin{equation}
\label{relat0}
 [A,T]=T
\end{equation}
and yields
\begin{equation}
\label{relat} [A,T^j]=jT^j,\quad j\geq 0.
\end{equation}
Let $E\in \mathrm{End}_{F[X]}(U)$ be given by
\begin{equation}
\label{ww4}
Eu_i=(X+(i-1))u_i,\quad 1\leq i\leq m.
\end{equation}
We see that $$TE=ET,\quad AE=EA.$$
Note that the subalgebra
$$
\E=F[E]=\{g(E)\,|\, g\in F[X]\},
$$
of $\mathrm{End}_F(U)$ is isomorphic to $F[X]/(p^{s_1})$.

It will be convenient to have a matrix form of the operators
$A,E,T$ of $U$. Given a monic polynomial
$f=f_0+f_1X+\cdots+f_{n-1}X^{n-1}+X^n\in F[X]$, its companion
matrix $C_f\in M_n(F)$ is defined by
$$
C_f=\left(%
\begin{array}{ccccc}
  0 & 0 & \cdots & 0 & -f_0 \\
  1 & 0 & \cdots & 0 & -f_1 \\
  0 & 1 & \cdots & 0 & -f_2 \\
  \vdots & \vdots & \cdots & \vdots & \vdots \\
  0 & 0 & \cdots & 1 & -f_{n-1} \\
\end{array}%
\right).
$$
We will use the abbreviated notation
$$
C_i=C_{p_i^{s_i}}\in M_{ds_i}(F),\quad 1\leq i\leq m.
$$
Let $\B_i$ be the $F$-basis of $U_i$ consisting of all $X^j u_i$
with $0\leq j<ds_i$. Then $\B=\B_1\cup\cdots\cup \B_m$ is an
$F$-basis of $U$ relative to which $A$ and $E$ are respectively
represented by
\begin{equation}
\label{defa}
\NA=\left(
             \begin{array}{cccc}
               C_1 & 0 & \dots & 0 \\
               0 & C_2 & \ddots & \vdots \\
               \vdots & \ddots & \ddots & 0 \\
               0 & \dots & 0 & C_m \\
             \end{array}
           \right)\in M_t(F)
\end{equation}
and
\begin{equation}
\label{defeg}
\NE=\left(
             \begin{array}{cccc}
               C_1 & 0 & \dots & 0 \\
               0 & C_2+I & \ddots & \vdots \\
               \vdots & \ddots & \ddots & 0 \\
               0 & \dots & 0 & C_m+(m-1)I \\
             \end{array}
           \right)\in M_t(F).
\end{equation}
Let $\NT\in M_t(F)$ be the matrix representing $T$ relative to
$\B$. Then by (\ref{defTT})
\begin{equation}
\label{deft} \NT=\left(
    \begin{array}{ccccc}
      0 & B_2 & 0 & \dots & 0 \\
      0 & 0 & B_3 & \ddots & \vdots\\
      \vdots & \ddots & \ddots & \ddots & 0 \\
      \vdots &  & \ddots & \ddots & B_{m}\\
      0 & \dots & \dots & 0 & 0 \\
    \end{array}
  \right),
\end{equation}
where
$$
B_i\in M_{ds_{i-1},ds_i}(F),\quad 2\leq i\leq m.
$$
From (\ref{relat0}) we obtain
\begin{equation}
\label{bsat} B_i(C_i+I)=C_{i-1}B_i,\quad 2\leq i\leq m.
\end{equation}
Moreover, (\ref{defTT}) implies that the $F[X]$-homomorphisms
$U_i^{\si(1)}\to U_{i-1}$, $2\leq i\leq m$, induced by $T$ are
injective, so each $B_i$ has nullity 0.

\begin{lemma}\label{gam} Given $\ga\in F$, a monic polyomial $f\in F[X]$ of degree $d\geq 1$, and integers $n\ge r\ge 1$,
let $S(\gamma,f,n,r)\in M_{dn,dr}(F)$ be the matrix defined by
\[
 S(\gamma,f,n,r)_{i,j}=\text{the coeff of $X^{i-1}$ in } (X+\gamma)^{j-1}f(X)^{n-r},\quad 1\le i\le dn,\; 1\le j\le dr,
\]
noting that $S(\ga,f,n,n)$ is the $dn\times dn$ Pascal matrix
$$
S(\ga,f,n,n)=\left(
    \begin{array}{cccccc}
      1 & \ga & \ga^2 &  \ga^3 & \dots &   \ga^{dn-1} \\
      0 & 1 & 2\ga & 3 \ga^2 & \dots & {{dn-1}\choose{1}}\ga^{dn-2}\\
       0 & 0 & 1 & 3 \ga & \dots & {{dn-1}\choose{2}}\ga^{dn-3}\\
        0 & 0 & 0 & 1 & \dots & {{dn-1}\choose{3}}\ga^{dn-4}\\
        \vdots & \vdots & \vdots &  & \ddots & \vdots\\
      0 & 0 & 0 & \dots & \dots & 1 \\
    \end{array}
  \right).
$$
Then
\begin{equation}
\label{abc}
 S(-\ga,f,n,r)(C_{f(X+\gamma)^r}+\gamma I)=C_{f(X)^n}S(-\ga,f,n,r).
\end{equation}
\end{lemma}

\begin{proof} Consider the vector spaces $V=F[X]/(f(X)^n)$ and
$W=F[X]/(f(X+\ga)^r)$.

Let $P$ (resp. $Q$) be the $F$-endomorphism
of $V$ (resp. $W$) given by multiplication by $X+(f(X)^n)$ (resp. by
$X+\ga+(f(X+\ga)^r)$),
and let $R:W\to V$ be the $F$-linear map defined by
$$
R:g(X)+(f(X+\ga)^r)\mapsto g(X-\ga)f(X)^{n-r}+(f(X)^n).
$$
Then the following diagram is commutative
\begin{equation}
\label{rpq}
\begin{matrix}
 F[X]/(f(X+\ga)^r)                        & \xrightarrow{\quad\displaystyle R\quad } &  F[X]/(f(X)^n) \\[2mm]
 Q\left\downarrow\rule{0cm}{.6cm}\right.  &                                          &   \left\downarrow\rule{0cm}{.6cm}\right.P   \\[3mm]
 F[X]/(f(X+\ga)^r)                        & \xrightarrow{\quad\displaystyle R\quad } &  F[X]/(f(X)^n) \\
 \end{matrix}
\end{equation}

Let $\B_V$ (resp. $\B_W$) denote the basis of $V$ (resp. $W$) obtained by projecting
the first $dn$ (resp. $dr$) elements of the canonical basis
$\{1,X,X^2,\dots\}$
of $F[X]$.

The matrix of $Q$ relative to $\B_W$ is $C_{f(X+\gamma)}+\gamma I$,
that of $P$ relative to $\B_V$ is $C_{f(X)}$, and that of $R$
relative to  $\B_V,\B_W$  is $S(-\ga,f,n,r)$. It
now follows from (\ref{rpq}) that (\ref{abc}) holds.
\end{proof}

We infer from Lemma \ref{gam} that if $B_2,\dots,B_m$ are as in (\ref{deft}) then
$$
B_i=S(-1,p_{i-1}, s_{i-1},s_i),\quad 2\leq i\leq m.
$$

\begin{lemma}\label{31} Let $1\leq j< m$. Then the linear epimorphism $F[X]\to T^j\E$ given by
$$
g(X)\mapsto T^jg(E)
$$
has kernel $(p^{s_{j+1}})$. In particular, $\dim_F T^j\E=ds_{j+1}$. Moreover, $T^m=0$.
\end{lemma}

\begin{proof} Use (\ref{defTT}) and (\ref{calct}) to calculate
$$
T^j u_1=\dots=T^j u_j=0,
$$
which is also valid when $j=m$, and
$$
T^j u_{j+k}=p_k(X)^{s_k-s_{j+k}}u_k,\quad 1\leq k\leq m-j.
$$
Set $p=p_1$ and suppose first that $T^j g(E)=0$. Then
$$
0=T^jg(E)u_{j+1}=g(X)p(X)^{s_1-s_{j+1}}u_1,
$$
so $p^{s_{j+1}}|g$. Conversely, if $p^{s_{j+1}}|g$ then $p^{s_{j+k}}|g$, so
$p(X+k-1)^{s_{j+k}}=p_k(X)^{s_{j+k}}$ divides $g(X+k-1)$ for $1\leq k\leq m-j$, whence
$$
0=p_k(X)^{s_k-s_{j+k}}g(X+k-1)u_k=T^j g(E)u_{j+k},\quad 1\leq k\leq m-j,
$$
that is, $T^j g(E)=0$.
\end{proof}

\begin{lemma}\label{33} The subspaces $\langle A\rangle,\E,T\E,\dots,T^{m-1}\E$ of $\End_F(U)$ are in direct sum.
\end{lemma}

\begin{proof} From $[A,T]=T$ we see that $A\notin\E$. The result now follows immediately
from the matrix representations of the $A,E$ and $T$.
\end{proof}
It follows from  (\ref{relat}), Lemma \ref{31} and Lemma \ref{33} that
$$\g=\langle A\rangle\ltimes F[E,T]=\langle A\rangle\ltimes(\E\oplus T\E\oplus\cdots\oplus T^{m-1}\E)$$
is a subalgebra of $\gl(U)$ of dimension $\dim_F(U)+1$ with 2-dimensional subalgebra
$$\s=\langle A\rangle\ltimes\langle T\rangle.$$

\begin{prop}\label{38} Let $\h$ be any subalgebra of $\g$ containing $\s$. Then
$U$ is a linked $\h$-module with nilpotency length $m$ and
nilpotency factors isomorphic to $U_1,\dots,U_m$ as
$F[X]$-modules.
 \end{prop}

\begin{proof} This follows from the very definitions of $A,E,T$ and~$U$.
\end{proof}

We will use the notation $\g(p_1,\dots,p_m)$ for $\g$ when $s_1=\cdots=s_m=1$.

\begin{prop}\label{39} Let $\h$ be any subalgebra of $\g$ containing $\s$. Then $U$ is a uniserial $\h$-module
if and only if $s_1=\cdots=s_m=1$.
\end{prop}

\begin{proof} Suppose first $s_1=\cdots=s_m=1$. It is then clear that $\socl(U)=U_1$. The subalgebra of $\gl(U/U_1)$
induced by $\g(p_1,\dots,p_m)$ is isomorphic to $\g(p_2,\dots,p_m)$. By induction, the socle series of $U/U_1$ is a composition series,
and hence so is that of~$U$, which means that $U$ is uniserial.

Suppose next that not all $s_i=1$. Then $s_1\geq 2$. Since the $F[X]$-homomorphism $U_2^{\sigma(1)}\to U_1$ induced by $T$ sends
$\socl(U_2^{\sigma(1)})$ to $\socl(U_1)$, it follows that
$$
((\socl(U_2^{\sigma(1)})\oplus \socl(U_1))/\socl(U_1))\oplus \socl(U_1/\socl(U_1))
$$
is contained in $\socl_\g(U/\socl_\g(U))$. Since $s_1\geq 2$ the second summand non-zero. Clearly so is the first,
so $U$ is not uniserial.
\end{proof}

\begin{prop}\label{310} Suppose that, if $F$ has prime characteristic $p$, then $p\geq m$
and the stabilizer of $p_1$ under the action (\ref{actf}) of $F^+$ is trivial.

Let $\Ze$ be the centralizer of $T$ in $\gl(U)$ and let ${\mathcal S}$ be the sum of all eigenspaces of $\ad\, A$ in $\gl(U)$.
Then
$$\Ze\cap {\mathcal S}=\E\oplus T\E\oplus\cdots\oplus T^{m-1}\E.$$
In particular, the $j$-eigenspace of $\ad\, A$ in $\Ze$ is $T^j\E$, $0\leq j<m$,
so the centralizer of $\s$ in $\gl(U)$ is equal to
$\E$.
\end{prop}

\begin{proof} It is clear that $\Ze$ and ${\mathcal S}$ are invariant
under $\ad\, A$ and thus $\Ze\cap {\mathcal S}$ is the direct sum of the
intersections of $\Ze$ with each of the eigenspaces
of $\ad\,A$ in $\gl(U)$.

Moreover, it is readily seen that $\Ze$ consists of block upper
triangular matrices with diagonal blocks of sizes $ds_1,\dots,d
s_m$.

Let $\delta\in F$ be an eigenvalue of $\ad\, A$ acting on $\gl(U)$
and let $H\in \Ze$ be a corresponding eigenvector represented by
$M\in\gl(t)$ relative to $\B$. Now $M$ must have a non-zero block,
say $D$, in some position $(k,\ell)$, where $k\leq \ell$. From
$$
D(C_\ell+\delta I)=C_kD
$$
we obtain a non-zero $F[X]$-homomorphism
$$
F[X]/p_\ell(X-\delta)^{s_\ell}\to F[X]/p_k(X)^{s_k},
$$
which implies
$$
p_\ell(X)=p_k(X+\delta).
$$
On the other hand,
$$
p_\ell(X)=p_{k}(X+\ell-k).
$$
Since $p_1,\dots,p_m$ are in the same $F^+$-orbit and $p_1$ has trivial stabilizer
(this is automatic if $\chr(F)=0$), so do $p_1,\dots,p_m$, which implies $\delta=\ell-k$. It follows that the
eigenvalues of $\ad\, A$ acting on $\Ze$ are $0,1,\dots,m-1$. Moreover,
since $p\geq m$ whenever $F$ has prime characteristic~$p$,
the matrix $M$ of every eigenvector $H\in \Ze$ of $\ad\, A$ of
eigenvalue $j$ has 0 blocks outside the $j$th superdiagonal. We
claim that $H=T^j g(E)$ for some $g\in F[X]$.

From $H(A+jI)=AH$ and the above description of $M$ we see that $H$
induces $m-j$ homomorphisms of $F[X]$-modules
$$
U_{j+k}^{\si(j)}\to U_k,\quad 1\leq k\leq m-j,
$$
and is completely determined by them. Since $p_k(X)^{s_{j+k}}$
annihilates $u_{j+k}\in U_{j+k}^{\si(j)}$ and the minimal
polynomial of $u_k\in U_k$ is $p_k(X)^{s_k}$, it follows that
\begin{equation}
\label{ufa} H u_{j+k}=p_k(X)^{s_k-s_{j+k}}g_k(X)u_k,\quad 1\leq
k\leq m-j,
\end{equation}
for some $g_1,\dots,g_{m-j}\in F[X]$. Thus, if $1<k\leq m-j$, we
have
$$
\begin{aligned}
THu_{j+k}&=T p_k(X)^{s_k-s_{j+k}}\,g_k(X)u_{k}\\
&=p_k(X-1)^{s_k-s_{j+k}}\,g_k(X-1)p_{k-1}(X)^{s_{k-1}-s_{k}}u_{k-1}\\
&=p_{k-1}(X)^{s_{k-1}-s_{k+j}}\,g_k(X-1)u_{k-1}.
\end{aligned}
$$
On the other hand,
$$
\begin{aligned}
HTu_{j+k}&=Hp_{j+k-1}(X)^{s_{j+k-1}-s_{j+k}}u_{j+k-1}\\
&=p_{j+k-1}(X-j)^{s_{j+k-1}-s_{j+k}}p_{k-1}(X)^{s_{k-1}-s_{j+k-1}}\,g_{k-1}(X)u_{k-1}\\
&=p_{k-1}(X)^{s_{k-1}-s_{k+j}}\,g_{k-1}(X)u_{k-1}.
\end{aligned}
$$
Since $[T,H]=0$ we deduce
$$
g_k(X-1)\equiv g_{k-1}(X)\mod p_{k-1}(X)^{s_{j+k}},
$$
which means
$$
g_k(X)\equiv g_{k-1}(X+1)\mod p_{k}(X)^{s_{j+k}}.
$$
Since $u_k$ is annihilated by $p_k(X)^{s_k}$ we may assume in
(\ref{ufa}) that
$$
g_k(X)=g_1(X+(k-1)),\quad 1\leq k\leq m-j.
$$
It follows from (\ref{ufa}) that $H=T^j g(E)$, where $g(X)=g_1(X-j)$.
\end{proof}

Consider the subspace $F[X]_n$, $n\geq 0$, of $F[X]$ defined by
$$
F[X]_n=\{a_0+a_1X+\cdots+a_{n-1}X^{n-1}\,|\, a_0,a_1,\dots,a_{n-1}\in F\}.
$$
\begin{theorem}\label{falta} Let $p_1\in F[X]$ be a monic irreducible polynomial.
Given $m>1$, define $p_2,\dots,p_m$ by means of (\ref{jio}), and let $s_1\geq\cdots\geq s_m$
be natural numbers. Let $U,A,T$ and $E$ be defined by (\ref{ww1})-(\ref{defTT}) and (\ref{ww4}).

Let $\h=\langle y\rangle \ltimes
\a$, where $\ad\, y$ acts diagonalizably on the abelian Lie algebra $\a$. For
$\ga\in F$ set
$$
\a(\ga)=\{v\in\a\,|\, [y,v]=\ga v\}.
$$
Suppose there exists a non-zero $v\in\a(1)$. Let
$$
g^j:\a(j)\to F[X]_{ds_{j+1}},\quad 0\leq j<m,
$$
be arbitrary linear maps satisfying
$g^1_v=1$. Then the map $R:\h\to \gl(U)$ given by
$$
\begin{array}{ll}
R(y)=A, & \\[1mm]
R(u)=g^j_u(E)T^j,  & u\in\a(j),\;0\leq j<m,\\[1mm]
R(u)=0,            & u\in\a(\ga),\; \ga\notin\{0,1,\dots,m-1\}
\end{array}
$$
defines a representation of $\h$ that
makes $U$ into a linked $\h$-module
of nilpotency length $m$. Moreover, $U$ is uniserial if and only
if $s_1=\cdots=s_m=1$.

Furthermore, suppose that, if $F$ has prime characteristic $p$, then $p\geq m$ and
the stabilizer of $p_1$ under the action (\ref{actf}) of $F^+$ is trivial.
Then any modification whatsoever to $m,p_1,s_1,\dots,s_m,g^0,\dots,g^{m-1}$ that keeps
$g^1_v=1$ produces an inequivalent representation.
\end{theorem}

\begin{proof} The first two statements follow from Propositions
\ref{38} and \ref{39} respectively. As for the third, $m$ is the nilpotency length of $U$, while
$p_1^{s_1},\dots,p_m^{s_m}\in F[X]$ are the prime powers associated to the nilpotency factors of $U$
when viewed as $F[X]$-modules via the action of $y$, Thus, any modification to $m,p_1,s_1,\dots,s_m$
yields an inequivalent representation. The assertion regarding $g^0,\dots,g^{m-1}$ follows from the last
statement of Proposition \ref{310}.
\end{proof}

\section{Classification of linked modules}

\begin{theorem}\label{primi} Suppose that $F$ is an infinite perfect field.
Consider the polynomial algebra $R=F[X_1,\dots,X_n]$
and assume that $U_1,\dots,U_m$ are uniserial $R$-modules. Then
there exists a normal linear combination $Y=a_1X_1+\cdots+a_n
X_n$, i.e., one where every $a_i\in F$ is non-zero, such that
$U_1,\dots,U_m$ are uniserial $F[Y]$-modules.
\end{theorem}

\begin{note}\label{exno}{\rm The case $m=1$ follows from \cite{CS2}, Theorem
3.1. Since the normality condition was not explicitly stated
there, we indicate the exact modifications required in the second
half of the proof of \cite{CS2}, Theorem 3.1, that ensure
normality. In the case $\ell=1$ the standard proof of the theorem
of the primitive element yields normality. In the inductive step,
by the same reason, for all but finitely many elements $\alpha\neq
0$ of $F$, the element $z=v+\al y$ is normal and satisfies
$B=F[z\vert_W]$.}
\end{note}

\begin{proof} By induction on $m$. The case $m=1$ follows from
\cite{CS2}, Theorem 3.1, as explained in Note \ref{exno}. Suppose
that $m>1$ and the result is true for less than $m$ uniserial
$R$-modules. Let $U_1,\dots,U_m$ be uniserial $R$-modules. By
the inductive hypothesis and the case $m=1$ there are normal
$Z,W\in\langle X_1,\dots,X_n\rangle$ such that $U_1,\dots,U_{m-1}$
are uniserial $F[Z]$-modules and $U_m$ is a uniserial
$F[W]$-module.

As explained in the proof of \cite{CS2}, Theorem 3.1, every
element of $R$ acts on a uniserial $R$-module with minimal
polynomial equal to a prime power. Let $p^k$ be the minimal
polynomial of $W$ acting on $U_m$, where $p\in F[X]$ is monic and
irreducible of degree $d$. Let $K$ be an algebraic closure of $F$.
We then have $p=(X-\al_1)\cdots (X-\al_d)$, where $\al_i\in K$ are
distinct. Let $E(\al_1),\dots,E(\al_d)$ be the generalized
eigenspaces of $W$ acting on $U_m$. For each $1\leq i\leq d$ there
is a basis of $E(\al_i)$ relative to which $W$ is represented by a
Jordan block $J_k(\al_i)$. Since $Z$ and $W$ commute, $Z$ must be
represented by a polynomial in $J_k(\al_i)$. It follows that there
is at most one $\be_i\in F$ such that the action of $Z+\be_i W$ on
$E(\al_i)$ is not represented, relative to some basis, by a Jordan
block $J_k(\ga_i)$. We deduce that the action of every $Z+\be W$,
$\be\in F$, on $U_m$ is represented by the direct sum of Jordan
blocks $J_k(\ga_1),\dots,J_k(\ga_n)$, with all $\ga_1,\dots,\ga_n$
distinct, except for finitely $\be\in F$. Thus, but for these
scalars, the minimal polynomial of the action of  $Z+\be W$ on
$U_m$ has degree $dk$. Since this minimal polynomial must be a
prime power, we see that $U_m$ is a uniserial $F[Z+\be W]$-module
for all but finitely many $\be\in F$. In spite of the apparent
asymmetry between $Z$ and $W$ given by location of $\be$, the same
argument shows that $U_j$, $1\leq j\leq m-1$, is a uniserial
$F[Z+\be W]$-module for all but finitely many $\be\in F$. Since
$Z+\be W$ is normal for all but finitely many $\be\in F$, the
result follows.
\end{proof}

Given $A\in M_m(F)$ and $B\in M_n(F)$, consider the endomorphisms
$\ell_A$ and $r_B$ of $M_{m,n}(F)$ given by
$$
\ell_A(Y)=AY,\quad r_B(Y)=YB.
$$
The minimal polynomial of $A$ will be denoted by $\mu_A$.

\begin{lemma}
\label{eigen} Let $A\in M_m(F)$ and $B\in M_n(F)$ and let $K$ be
an algebraic closure of $F$.  Then every eigenvalue of
$\ell_A-r_B$ in $K$ is of the form $\al-\be$, where $\al$ is a
root of $\mu_A$ and $\be$ is a root of $\mu_B$.
\end{lemma}

\begin{proof} We have $\mu_{\ell_A}=\mu_A$ and $\mu_{r_B}=\mu_B$. Since $\ell_A$ and $r_B$ commute,
they can be simultaneously triangularized over $K$, whence the
result follows.
\end{proof}

\begin{cor}
\label{eigen2} Let $A\in M_m(F)$ and $B\in M_n(F)$. Suppose that
$\gcd(\mu_A,\mu_B)=1$. Then $\ell_A-r_B$ is an isomorphism.
\end{cor}

\begin{lemma}\label{up} Let $A$ be a block upper triangular matrix over
an arbitrary field $F$, with diagonal blocks $A_1,\dots,A_m$, not
necessarily of the same size. Then there is a block upper
triangular matrix $P$ over $F$ with identity diagonal blocks such
that $B=P^{-1}AP$ satisfies: if $\gcd(\mu_{A_i},\mu_{A_j})=1$ then
block $(i,j)$ of $B$ is 0.
\end{lemma}

\begin{proof} We will clear all required blocks of $A$ by means of transformations $$T\mapsto
Q^{-1}TQ,$$ where $Q=Q(i,j,Z)$ has identity diagonal blocks, block
$(i,j)$ with $i<j$ is equal to $Z$, and all other blocks are equal
to 0.

Suppose that $C$ was obtained from $A$ by means of a sequence of these
transformations and satisfies the following: for some $1\leq i<n$,
all required blocks of $C$ are 0 below row block $i$, and there is
$i<j\leq n$ such that if $i<k<j$ and $\gcd(\mu_{A_i},\mu_{A_k})=1$
then block $(i,k)$ of $C$ is 0. Let $U$ be the $(i,j)$ block of
$C$. By Lemma \ref{eigen2}, if $\gcd(\mu_{A_i},\mu_{A_j})=1$,
there is a unique $Z$ such that
\begin{equation}
\label{zet} A_iZ-ZA_j=-U.
\end{equation}
Let $Q$ be the identity matrix if $\gcd(\mu_{A_i},\mu_{A_j})\neq
1$, and $Q=Q(i,j,Z)$, with $Z$ satisfying (\ref{zet}), otherwise.
Then the blocks of $D=Q^{-1}CQ$ and $C$ coincide below row block
$i$ as well as within row block $i$ but to the left of block
$(i,j)$. Moreover, block $(i,j)$ of $D$ is 0 if
$\gcd(\mu_{A_i},\mu_{A_j})=1$. We may thus continue this process
and find $B$ as required.
\end{proof}

For the remainder of this section let $\g$ be a Lie algebra with solvable radical $\r$ and set
$\n=[\g,\r]$.

\begin{lemma}
\label{nilradical} Suppose that $\chr(F)=0$. Then $\n$ acts
trivially on every irreducible $\g$-module~$V$.
\end{lemma}

\begin{proof} Suppose first that $F$ is algebraically closed. By Lie's theorem
there is a linear functional $\lambda:\n\to F$ such that
$$
U=\{v\in V\,|\, xv=\lambda(x)v\text{ for all }x\in\n\}
$$
is non-zero. Since $\n$ is an ideal of $\g$, the Invariance Lemma
(\cite{FH}, Lemma 9.13) ensures that $U$ is a $\g$-submodule of
$V$, so $U=V$ by irreducibility. Thus $\n$ acts on $V$ by scalar
operators. But $\n$ acts on $V$ via traceless operators, so the
result follows.

Suppose next $F$ is arbitrary. By considering an algebraic closure
of $F$ we deduce from the above case that $\n$ acts via nilpotent
operators on $V$. By Engel's theorem,
$$
W=\{v\in V\,|\, xv=0\text{ for all }x\in\n\}
$$
is non-zero. We deduce, as above, that $\n$ acts trivially on $V$.
\end{proof}

Let $U$ be an admissible $\g$-module with nilpotency
length $m>1$, which means that $\n$ does not act trivially on $U$.
The nilpotency factors of $U$ will be denoted by
$$
U_i=U^{(i)}/U^{(i-1)},\quad 1\leq i\leq m.
$$
Given $2\leq i\leq m$, every $v\in\n$
gives rise to an element $f^i_v\in \Hom_F(U_i,U_{i-1})$ defined by
$$
f^i_v(u+U^{(i-1)})=vu+U^{(i-2)},\quad u\in U^{(i)}.
$$
Both $\n$ and $\Hom_F(U_i,U_{i-1})$ are $\g$-modules, and we see
that the map $$f^i:\n\to \Hom_F(U_i,U_{i-1})$$ defined by
$$v\mapsto f^i_v$$ is a homomorphism of $\g$-modules. The very definition of $U^{(0)},\dots,U^{(m)}$ implies
the following:
\begin{equation}
\label{nocero}\text{Given any }u\in U^{(i)}\setminus
U^{(i-1)}\text{ there exists }v\in\n\text{ such that
}f_v^i(u+U^{(i-1)})\neq 0.
\end{equation}

We shall make the following assumptions for the remainder of this section.

\medskip

{\sc (A1)} $U$ is linked.

\medskip

{\sc (A2)} $\g=\langle x\rangle\ltimes \a$,
where $\ad\, x$ acts diagonalizably on the abelian Lie algebra $\a$.

\smallskip

\noindent Thus $\n=[\g,\g]$ is the sum
of all eigenspaces of $\ad\, x$ in $\a$ with non-zero eigenvalue.
Let $\a_0$ the the 0-eigenspace of $\ad\, x$ in $\a$, so that
$\g=(\langle x\rangle\oplus \a_0)\ltimes\n$.

\medskip

{\sc (A3)} $F$ is an infinite perfect field or $\a_0=0$.

\begin{lemma}\label{r1} There is $y\in (\langle x\rangle\oplus \a_0)\setminus\a_0$
such that $U_1,\dots,U_m$ are
uniserial modules $F[X]$-modules via $y$.
\end{lemma}

\begin{proof} If $\a_0=0$ take $y=x$. If $F$ is an infinite perfect field
use Theorem \ref{primi}.
\end{proof}

Clearly $\g=\langle y\rangle\ltimes \a$. Since $\ad\, a=0$ for all $a\in\a_0$,
we see that $\ad\, y$ acts diagonalizably on $\a$ and has non-zero eigenvalues on $\n$.

By Lemma \ref{r1}, there exist $m$ monic irreducible polynomials $p_1,\dots,p_m\in F[X]$
such that
\begin{equation}
\label{expo}
U_1\cong F[X]/p_1^{s_1},\dots,U_m\cong F[X]/p_m^{s_m},
\end{equation}
as $F[X]$-modules. Given $\ga\in F$ and $1\leq i\leq m$, we can make $U_i$
into an $F[X]$-module, denoted by $U_i^\ga$, via
$$
g(X)\cdot w=g(X+\ga)w,\quad w\in U_i.
$$
For $\ga\in F$, we further let
$$
\n(\ga)=\{v\in\n\,|\, [y,v]=\ga v\}.
$$
The fact that $f^2,\dots,f^m$ are homomorphisms of $\g$-modules implies that
\begin{equation}
\label{gahom} f^{i}_v:U_i^\ga\to U_{i-1},\quad 2\leq i\leq m,\ga\in F,v\in\n(\ga),
\end{equation}
is a homomorphism of $F[X]$-modules.

\begin{lemma}\label{r2} Corresponding to any $2\leq i\leq m$ there exist $\ga_i\in F$, $\ga_i\neq 0$,
and $v_i\in \n(\ga_i)$ such that $f^i_v$ is injective.
\end{lemma}

\begin{proof} It is clear that $U_i$ is a uniserial $F[X]$-module with irreducible socle
$$
\mathrm{soc}(U_i)=\{u\in U_i\,|\, p_i(X)^{s_i-1}u=0\}.
$$
This is exactly the socle of $U_i^\ga$ for every $\ga\in F$. Suppose, if possible, that for every $\ga\in F$ and every $v\in\n(\ga)$, the
$F[X]$-homomorphism $f^i_v:U_i^\ga\to U_{i-1}$ fails to be injective. Then $\mathrm{soc}(U_i)$ is contained in the kernel of $f^i_v:U_i^\ga\to U_{i-1}$ for every $\ga\in F$ and $v\in\n(\ga)$. Since $\ad\, y$ acts diagonalizably on $\n$, condition (\ref{nocero}) is violated. Given that 0 is not an eigenvalue of $\ad\, y$ on $\n$, the result follows.
\end{proof}

It will be convenient now to adopt a matrix point of
view. Let $\B=\B_1\cup\cdots\cup\B_m$ be a basis of $U$ such that
$\B_1\cup\cdots\cup\B_i$ is a basis of $U^{(i)}$ for all $1\leq i\leq m$. 
Then,
relative to $\B$, every element $z$ of $\g$ is represented by a
block upper triangular matrix $M(z)$, with diagonal blocks of
sizes $ds_1,\dots,ds_m$, where $d=\deg(p_1)$. Moreover, the
diagonal blocks of $M(y)$ are $C_1,\dots,C_m$, where $C_i$ is the
companion matrix of~$p_i^{s_i}$, $1\leq i\leq m$, and the diagonal
blocks of $M(v)$ are equal to 0 for every $v\in\n$. Furthermore,
given any $2\leq i\leq m$, there is an eigenvector $v_i$ of $\ad\, y$
in $\n$ such that block $(i-1,i)$ of $M(v_i)$ has nullity 0.

\begin{lemma}\label{r3} Let $v=v_2$. Then $f^2_v,\dots,f^m_v$
are injective.
\end{lemma}

\begin{proof} Suppose, if possible, that this fails. Let $L_2,\dots,L_{m}$ be
the blocks along the first superdiagonal of $M(v)$. Then there is
$2\leq j\leq m-1$ such that $L_2,\dots,L_j$ has nullity 0 but
$L_{j+1}$ has nullity $>0$. By above, there is $u\in\n$ such that
block $(j,j+1)$ of $M(u)$ has nullity 0. Let $Q_2,\dots,Q_{m}$ be
the blocks along the first superdiagonal of $M(u)$, so that
$Q_{j+1}$ has nullity 0. Then block $(j,j+2)$ of $M(v)M(u)$ is
$L_jQ_{j+1}$, which has nullity 0, while block $(j,j+2)$ of
$M(u)M(v)$ is $Q_jL_{j+1}$, which has nullity $>0$. This
contradicts the fact that $M(v)M(u)=M(u)M(v)$.
\end{proof}

We have $[y,v]=\ga v$ for a unique $\ga\in F$, $\ga\neq 0$. We henceforth replace $y$ by $\gamma^{-1}y$.

\begin{lemma}\label{r4} The polynomials $p_1,\dots,p_m$ are related by
\begin{equation}
\label{polu} p_j(X)=p_{1}(X+j-1),\quad 1\leq j\leq m,
\end{equation}
and their exponents in (\ref{expo}) satisfy
$$
s_1\geq s_2\geq\cdots\geq s_m.
$$
\end{lemma}

\begin{proof} Clearly
$$U_i^1\cong F[X]/p_{i}(X-1)^{s_i},\quad 2\leq i\leq m.$$
Since $f^i_v:U_i^1\to U_{i-1}$ is a non-zero $F[X]$-homomorphism, we have
$$p_{i}(X-1)=p_{i-1}(X),\quad 2\leq i\leq m.$$
Since $f^i_v$ is, in fact, injective, $s_{i-1}\geq s_i$ for all $2\leq i\leq m$.
\end{proof}

\begin{lemma}\label{r5} Let ${\mathcal P}$ be the set consisting of all distinct polynomials amongst $p_1,\dots,p_m$.
Then

(1) If $\chr(F)=0$ then $|{\mathcal P}|=m$, i.e., $p_1,\dots,p_m$ are
all distinct from each other.

(2) Suppose $F$ has prime characteristic $p$ and restrict to $\Z_p^+$ the action (\ref{actf}) of $F^+$ on $F[X]$. Then ${\mathcal P}$ is $\Z_p^+$-stable and the action of $\Z_p^+$ on ${\mathcal P}$ is transitive. Moreover, each stabilizer is trivial or $\Z_p^+$.
In the latter case ${\mathcal P}$ is a singleton and in the former $|{\mathcal P}|=p$, with
$p_i=p_j\Leftrightarrow i\equiv j\mod p$. In particular, if $p\geq m$ and $|{\mathcal P}|>1$ then $p_1,\dots,p_m$ are
all distinct from each other.
\end{lemma}

\begin{proof} This is clear from (\ref{polu}).
\end{proof}


We make the following assumption from now on.

\medskip

{\sc (A4)} If $F$ has prime characteristic $p$ then $p\geq m$ and the stabilizer of $p_1$ under the action (\ref{actf}) of $F^+$
on $F[X]$ is trivial.

\medskip

We infer from (A4) and Lemma \ref{r5} that $p_1,\dots,p_m$ are distinct from each other.

\begin{lemma}\label{r6} There is a basis of $U$ relative to which $y$ is represented by
$\NA$, as in (\ref{defa}), and $v$ is represented by
$$
\NT'=\left(
    \begin{array}{ccccc}
      0 & L_2 & 0 & \dots & 0 \\
      0 & 0 & L_3 & \ddots & \vdots\\
      \vdots & \ddots & \ddots & \ddots & 0 \\
      \vdots &  & \ddots & \ddots & L_{m}\\
      0 & \dots & \dots & 0 & 0 \\
    \end{array}
  \right),
$$
where $L_i$ has nullity 0 and satisfies $L_i(C_i+I)=C_{i-1}L_i$.
\end{lemma}

\begin{proof} Since $p_1,\dots,p_m$ are all distinct, this follows from Lemmas \ref{up}
and \ref{r3}.
\end{proof}

\begin{lemma}\label{r7} There is a basis of $U$ relative to which $y$ is represented by
$\NA$, as in (\ref{defa}), and $v$ is represented by $\NT$, as in (\ref{deft}).
\end{lemma}

\begin{proof} Lemma \ref{r6} and (\ref{bsat}) imply that $B_i$ and $L_i$ both represent a
monomorphism of $F[X]$-modules
$$F[X]/p_{i-1}^{s_i}\to F[X]/p_{i-1}^{s_{i-1}},\quad 2\leq i\leq m.$$ It follows that  $L_i=B_iq_i(C_i)$ for
some $q_i\in F[X]$ satisfying $\gcd(q_i,p_i)=1$. We claim that
there exist $g_1,\dots,g_m\in F[X]$ such that
$$
P=\left(
             \begin{array}{cccc}
               g_1(C_1) & 0 & \dots & 0 \\
               0 & g_1(C_2) & \ddots & \vdots \\
               \vdots & \ddots & \ddots & 0 \\
               0 & \dots & 0 & g_m(C_m) \\
             \end{array}
           \right),\quad \gcd(g_i,p_i)=1,
$$
and
$$
P \NT' P^{-1}=\NT.
$$
This means
$$
g_i(C_i)B_{i+1}q_{i+1}(C_{i+1})g_{i+1}(C_{i+1})^{-1}=B_{i+1},\quad
1\leq i<m.
$$
Start by setting $g_1=1$, $h_2=1$ and $g_2=h_2q_2$. By
(\ref{bsat})
$$
g_2(C_2)B_3=B_3g_2(C_3+I)=B_3h_3(C_3),
$$
where $h_3(X)=g_2(X+1)$ is relatively prime to $p_3(X)=p_2(X+1)$.
Next set $g_3=h_3q_3$, and repeat. This proves the claim.
Obviously, we also have $P\NA P^{-1}=\NA$.
\end{proof}

We are finally in a position to classify all linked $\g$-modules.
Since a linked $\g$-module annihilated by $\n$ is nothing but a
uniserial module over the abelian Lie algebra $\g/\n$, a case which was considered in \cite{CS2},
we may restrict our attention to linked
$\g$-modules of nilpotency length $m>1$.

\begin{theorem}\label{main} Consider the Lie algebra $\g=\langle x\rangle\ltimes \a$,
where $\ad\, x$ acts diagonalizably on the abelian Lie algebra
$\a$. Set $\n=[\g,\g]$, namely the sum of
all eigenspaces of $\ad\, x$ in $\a$ with non-zero eigenvalue, and
let $\a_0$ the the 0-eigenspace of $\ad\, x$ in $\a$.

Suppose that $F$ is an infinite perfect field or $\a_0=0$, and let $U$ be a
linked $\g$-module of nilpotency length $m>1$. Then

\medskip

(1) There is $y\in(\langle x\rangle\oplus\a_0)\setminus\a_0$
such that all the nilpotency factors $$U^{(1)}/U^{(0)},\dots,U^{(m)}/U^{(m-1)}$$
of $U$ are uniserial $F[X]$-modules via the action of $y$, and if $p_1,\dots,p_m\in F[X]$ are the monic irreducible
polynomials whose powers $p_1^{s_1},\dots,p_m^{s_m}$ satisfy
$$
U^{(i)}/U^{(i-1)}\cong F[X]/(p_i)^{s_i},\quad 1\leq i\leq m,
$$
as $F[X]$-modules via $y$, then $p_1,\dots,p_m$ have common degree $d$, are related by
$$
p_i(X)=p_1(X+i-1),\quad 1\leq i\leq m,
$$
and the exponents $s_1,\dots,s_m$ satisfy
$$
s_1\geq\cdots\geq s_m.
$$
Assume that, if $F$ has prime characteristic $p$, then $p\geq m$ and the stabilizer of $p_1$ under the action (\ref{actf}) of $F^+$
on $F[X]$ is trivial.

\medskip

(2) For $\delta\in F$ let
$$
\a(\delta)=\{u\in \a\,|\, [y,u]=\delta u\}
$$
and set
$$
S=\{\delta\in F\,|\, \a(\delta)U\neq 0\}.
$$
Then $1\in S$ and every $j\in S$ is an integer satisfying $0\leq
j<m$.

\medskip

(3) There is basis of $U$ such that:

\noindent $\bullet$ $y$ is represented by $\NA$ as in
{\rm (\ref{defa})}.

\noindent $\bullet$ There is $v\in\n$ such that $v\neq 0$,
$[y,v]=v$ and $v$ is represented by $\NT$ as in {\rm (\ref{deft})}.

\noindent $\bullet$ Corresponding to every integer $j$ satisfying $0\leq j<m$ there is linear
map
$$
g^j:\a(j)\to F[X]_{ds_{j+1}}
$$
such that $g^1_v=1$ and every $u\in \a(j)$ is represented by $T^j
g^j_u(\NE)$, with $\NE$ as in {\rm (\ref{defeg})}. In particular,  the
image of $\g$ in $\gl(t)$, where
$$
t=d(s_1+\cdots+s_m),
$$
is
\begin{equation}
\label{imagen} \langle A\rangle\ltimes(\E_0\oplus\E_1
\NT\oplus\cdots\oplus \E_{m-1} \NT^{m-1}),
\end{equation}
where
$$
\E_j=\{g^j_u(\NE)\,|\, u\in\n(j)\},\quad 0\leq j\leq m.
$$

\medskip

(4) $U$ is uniserial if and only if $s_1=\cdots=s_m=1$, in which
case $y$ must act semisimply on $U$, and hence so does $x$ if $\a_0=0$ or $d=1$.
\end{theorem}

\begin{proof} The lemmas preceding the theorem justify (1) as well
as the first two assertions of (3), while (2) and the third assertion of (3)
follow from Lemma \ref{31} and Proposition \ref{310}.
Moreover, (4) is a consequence of Proposition \ref{39}.
\end{proof}

\begin{note}{\rm Combining Theorems \ref{falta} and \ref{main} yields a complete classification
of linked $\g$-modules when $\chr(F)=0$ and a partial classification when $F$ has prime characteristic.
Further information on the latter case can be found in \cite{CS4}.}
\end{note}

\begin{cor}\label{sinno} Keep the hypotheses and notation of Theorem \ref{main} and suppose, in addition,
that $U$ is faithful. Then {\em every} eigenvalue $j$ of $\ad\,y$
acting on $\a$ must an integer satisfying $0\leq j<m$, and the
multiplicity of $j$ is bounded above by $ds_{j+1}$.

Assume, furthermore, that $U$ is uniserial and that $p=X-\alpha$
for some $\alpha\in F$, which is necessarily the case if $F$ is
algebraically closed. Then all eigenvalues $\ad\, y$ acting on $\a$
must have multiplicity 1. Moreover, there is a basis of $U$
relative to which $y$ is represented by
$$
\NA=\left(
             \begin{array}{cccc}
               \alpha & 0 & \dots & 0 \\
               0 & \alpha-1 & \ddots & \vdots \\
               \vdots & \ddots & \ddots & 0 \\
               0 & \dots & 0 & \alpha-(m-1) \\
             \end{array}
           \right)\in\gl(m),
$$
some eigenvector $v\in\n$ of $\ad\, y$ with eigenvalue 1 is
represented by
$$
\NT=\left(
    \begin{array}{ccccc}
      0 & 1 & 0 & \dots & 0 \\
      0 & 0 & 1 & \ddots & \vdots\\
      \vdots & \ddots & \ddots & \ddots & 0 \\
      \vdots &  & \ddots & \ddots & 1\\
      0 & \dots & \dots & 0 & 0 \\
    \end{array}
  \right)\in\gl(m),
$$
and if $(v_j)_{j\in S}$ is a basis of $\a$ formed by eigenvectors
of $\ad\, y$, subject to $v_1=v$, then every $v_j$ is represented
by $\beta_j \NT^{j}$, where $0\neq\beta_j\in F$ and $\beta_1=1$.
Moreover, the isomorphism type of $U$ is uniquely determined by
$m,\alpha$ and $(\be_j)_{j\in S}$.
\end{cor}

\begin{note}{\rm Corollary \ref{sinno} shows that not every Lie algebra has a
faithful uniserial module.}
\end{note}





\begin{note}\label{n5}{\rm If (A3) or (A4) do not hold then
the above classification of linked modules fails.

Indeed, suppose first that $F$ is finite or imperfect, and that
$\a=\a_0\neq 0$. Then $\g$ is abelian of dimension $>1$ and there
exists a a uniserial $\g$-module, and hence admissible, that is
not a uniserial $F[y]$-module for any $y\in\g$ (see \cite{CS3},
Theorem 2.7 and Note 3.5).

Suppose next that $F$ has prime characteristic $p$ and that $f\in
F[X]$ is an irreducible polynomial satisfying $f(X)=f(X+1)$. The
irreducible Artin-Schreier polynomials, as well as a generalized
version of them studied in \cite{GS}, satisfy this property. Let
$S=S(-1,f,1,1)$, in the notation of Lemma \ref{gam}, and set
$$M=\left(%
\begin{array}{cc}
  C_f & C_f \\
  0 & C_f \\
\end{array}%
\right),\; N=\left(%
\begin{array}{cc}
  0 & S \\
  0 & 0 \\
\end{array}%
\right).
$$
It is easy to verify that the minimal polynomial of $M$ is $f^2$
(see \cite{GS2}, Lemma 4.1, for details). Moreover, by Lemma
\ref{gam},
$$
[C_{f},S]=S,
$$
whence
$$
[M,N]=N.
$$
Let $\g=\langle x\rangle\ltimes \langle v\rangle$, where
$[x,v]=v$. Then $x\mapsto M$ and $v\mapsto N$ yields a faithful
uniserial $\g$-module $U$ upon which $v$ acts nilpotently.
However, the polynomials (\ref{polu}) associated to the action of
$x$ on the nilpotency factors of $U$ produce a single elementary
divisor for the action of $x$ on the whole of $U$. In particular,
$x$ does not act semisimply on $U$. These last two features are
impossible under the hypothesis of Theorem \ref{main}.

Even when $F$ is algebraically closed of prime characteristic
$p<m$, the above classification of linked modules of nilpotency
length $m$ still fails. Details, including a correct
classification, can be found in \cite{CS4}. }
\end{note}


\end{document}